\documentclass[14pt,a4paper]{article}

\usepackage[utf8]{inputenc}
\usepackage[T2A]{fontenc}
\usepackage[english,russian]{babel}

\usepackage{longtable}
\usepackage{epsfig,array}

\usepackage{mathtools}
\mathtoolsset{showonlyrefs}

\usepackage{indentfirst} 
\usepackage[14pt]{extsizes}
\usepackage{amssymb}
\usepackage{amsmath}
\usepackage{amsthm}
\usepackage{pifont}
\usepackage{subfig}
\usepackage{graphicx}
\usepackage[colorinlistoftodos]{todonotes}
\usepackage[colorlinks=true, allcolors=blue]{hyperref}
\usepackage{algorithm}
\usepackage{algorithmic}

\newtheorem{theorem}{Теорема}

\newtheorem{lemma}{Лемма}

\def\UDK#1{{\leftline{УДК {#1}}}}

\def\ag#1{{\color{black}#1}}

\def\dk#1{{\color{blue}#1}}

\usepackage{geometry}
\def \NN {\mathbb N}

\newcommand{\argmin}{\mathop{\arg\!\min}}

\newcommand{\la}{\langle}
\newcommand{\ra}{\rangle}

\def \R {\mathbb R}

\begin{document}
\renewcommand{\abstractname}{\vspace{-\baselineskip}}

$$
\\\\
$$

{\it \UDK{519.853.62}}

\begin{center}
\textbf{Ускоренный метаалгоритм для задач выпуклой оптимизации}

\textbf{А.\,В.~Гасников$^{1,2}$, Д.\,М.~Двинских$^{2,1,3}$, П.\,Е.~Двуреченский$^{3,2}$, Д.\,И.~Камзолов$^{1}$, В.\,В.~Матюхин$^{1}$, Д.\,А.~Пасечнюк$^{1}$, Н.\,К.~Тупица$^{1}$, А.\,В.~Чернов$^{1}$}

\textit{
$^{1}$Московский физико-технический институт (национальный исследовательский университет), Долгопрудный, Московская обл., Россия \\
$^{2}$Институт проблем передачи информации им. А.А. Харкевича Российской академии наук, Москва, Россия \\
$^{3}$Институт прикладного анализа и стохастики им. К. Вейерштрасса, Берлин, Германия
\\
{$^*$ e-mail: kamzolov.dmitry@phystech.edu}}

{Поступила в редакцию: 18 апреля 2020.\\
Переработанный вариант  ...............................\\Принята к публикации 24 августа 2020.}

\end{center}

\begin{abstract}
\noindent В работе предлагается \ag{оболочка, названная} <<ускоренный метаалгоритм>>, котор\ag{ая} позволяет \ag{единообразно} получать ускоренные методы решения задач выпуклой безусловной минимизации в различных постановках \ag{на базе неускоренных вариантов}. В качестве приложений \ag{приводятся} квазиоптимальные алгоритмы для минимизации гладких функций с Липшицевыми производными произвольного порядка, а также для решения гладких минимаксных задач. Предложенн\ag{ая оболочка} является более общ\ag{ей}, чем существующие, а также позволяет получать лучшие оценки скорости сходимости и практическую эффективность для ряда постановок задач.

\textbf{Ключевые слова}: выпуклая оптимизация, проксимальный ускоренный метод, тензорные методы, неточный оракул, слайдинг, каталист
\end{abstract}

\section{Введение} \label{section_1}

В последние 15 лет в численных методах гладкой выпуклой оптимизации преобладают, так называемые, ускоренные методы. Прообразом таких методов является метод тяжелого шарика Б.Т.~Поляка и моментный метод Ю.Е.~Нестерова \cite{Gasnikov2018, NesterovLectures}.
Оказалось, что для многих задач гладкой выпуклой оптимизации оптимальные методы (с точки зрения числа вычислений градиента функции; в общем случае, старших производных) могут быть найдены среди ускоренных методов \cite{Gasnikov2018, NesterovLectures, Lan2019}. 
Появилось огромное число работ, в которых предлагаются различные варианты ускоренных методов для разных классов задач, см., например, обзор литературы в \cite{Gasnikov2018, Lan2019}. Каждый раз процедура ускорения принимала свою причудливую форму. Естественно, возникло желание как-то унифицировать все это. 
В 2015 году это было сделано для широкого класса (рандомизированных) градиентных методов с помощью проксимальной ускоренной оболочки, названной Каталист\footnote{\ag{Здесь и далее в качестве названий подходов / алгоримтов иногда будут использоваться англицизмы. Дело в том, что дословный перевод исходно английских выражений на русский язык может только запутывать дело.}

\ag{Отметим также, что под <<проксимальной оболочкой>> здесь и далее имеется в виду просто проксимальный алгоритм. Слово <<оболочка>> подразумевает, что в проксимальном алгоритме на каждой итерации имеется своя внутренняя (вспомогательная) задача оптимизации, которую, как правило, нельзя решить аналитически. Ее нужно решать численно. Поэтому внешний проксимальный метод можно понимать как <<оболчку>> для метода, использующегося для решения внутренней задачи.}}  \cite{Catalyst}.
С 2013 года данные результаты стали активно переноситься на тензорные методы (использующие старшие производные) \cite{Doikov2019, Gasnikov2019, Monteiro2013, Nesterov2020}. В самое последнее время предпринимаются попытки унификации процедур ускорения  для седловых задач и задач со структурой (композитных задач) \cite{Alkousa2019, Ivanova2020, Kamzolov2020, Lin2020}. Во всех этих направлениях, по-прежнему, использовалось значительное разнообразие  ускоренных проксимальных оболочек \cite{Gasnikov2018, Catalyst, Doikov2019, Gasnikov2019, Monteiro2013, Nesterov2020, Ivanova2020, Kamzolov2020,  Gasnikov2019COLT, Bubeck2020, Jiang2020,  Ivanova2019}. 
Метод из данной работы\footnote{\ag{Строго говоря, это даже не метод (алгоритм), а скорее оболочка (в смысле определенном выше). В данной статье было выбрано название <<ускоренный метаалгоритм>>. Первое слово поясняет цель разрабатываемой оболочки -- ускорение метода, использующегося в качестве базового (решающего внутреннюю задачу). Однако, в отличие от стандартной (ускоренной) оболочки, в предложенной в данной статье оболочке все же в ряде важных случаев вспомогательная задача решается аналитически и, стало быть, говорить об этой оболочке, как <<оболочке>>, а не как об обычном алгоритме, не совсем корректно. Поэтому было решено использовать более нейтральное в этом смысле слово -- <<метаалгоритм>>.}} будет во многом базироваться на схеме из \cite{Bubeck2020}.

В данной работе показывается, что достаточно изучить всего одну ускоренную проксимальную оболочку, которая позволяет получать все известные нам ускоренные методы для задач гладкой выпуклой безусловной оптимизации. Причем в ряде случаев предложенный ускоренный метаалгоритм позволяет убирать логарифмические зазоры в оценках сложности (по сравнению с нижними оценками), имевшие место в предыдущих подходах.

\section{Основные результаты}\label{section_2}	

Рассмотрим следующую задачу ($x_*$ -- решение задачи)
\begin{equation}
\label{eq1}
\min\limits_{x \in \R^d }\{ F\left( x \right):=f\left( x \right)+g\left( x \right)\} ,
\end{equation}
где $f$ и $g$ выпуклые функции.

Везде в дальнейшем под $\|\,\cdot\,\|$ будем понимать обычную евклидову норму в пространстве $\R^d$, $$D^k f(x)[h]^k = \sum_{i_1,...,i_d \ge 0:\,\, \sum_{j=1}^d i_j = k} \frac{\partial^k f(x)}{\partial x_1^{i_1}...\partial x_d^{i_d}}h_1^{i_1} \cdot...\cdot h_d^{i_d},$$  $$\|D^k f(x)\| = \max_{\|h\|\le 1} \left\|D^k f(x)[h]^k\right\|.$$
Будем считать, что $f$ имеет Липшицевы производные порядка $p$ ($p \in \NN$):
\begin{equation}
    \|D^p f(x)- D^p f(y)\|\leq L_{p,f}\|x-y\|.
    \label{def_lipshitz}
\end{equation}
Здесь и далее (см., например, \eqref{unif_conv}) можно считать, что $x,y \in \R^d$ принадлежат евклидову шару с центром в точке $x_{*}$ и радиусом $O(\|x_0 - x_*\|)$, где $x_0$ -- точка старта \cite{Gasnikov2019}.

Введем аппроксимацию рядом Тейлора функции $f$:
\begin{equation}
    \Omega_{p}(f,x;y)=f(x)+\sum_{k=1}^{p}\frac{1}{k!}D^{k}f(x)\left[ y-x \right]^k, y\in \R^d.
    \label{eq_taylor}
\end{equation}
Заметим, что из \eqref{def_lipshitz} следует \cite{NesterovImplementable}
\begin{equation}
   \left|f(y)-\Omega_{p}(f,x;y)\right| \leq \frac{L_{p,f}}{(p+1)!}\|y-x\|^{p+1}.
    \label{eq_sumup}
\end{equation}

\begin{algorithm} [h!]
\caption{Ускоренный Метаалгоритм (УМ) (УМ($x_0$,$f$,$g$,$p$,$H$,$K$))}\label{alg:highorder}
	\begin{algorithmic}[1]
		\STATE \textbf{Input:} $p \in \NN$, $f : \R^d \rightarrow \R$, 
		$g : \R^d \rightarrow \R$, $H > 0$. 
		\STATE $A_0 = 0, y_0 = x_0$.
		\FOR{ $k = 0$ \TO $k = K- 1$}
		\STATE Определить пару $\lambda_{k+1} > 0$ и $y_{k+1}\in \R^d$ из условий
		\[
\frac{1}{2} \leq \lambda_{k+1} \frac{H \|y_{k+1} - \tilde{x}_k\|^{p-1}}{p!}  \leq \frac{p}{p+1} \,,
\]
где
\begin{equation}
\label{prox_step}
\hspace{-2em} y_{k+1} = \argmin_{y\in \R^d} \left\{\widetilde{\Omega}^k(y):=
\Omega_{p}(f,\tilde{x}_k;y)+g(y) +\frac{H}{(p+1)!}\|y-\tilde{x}_k\|^{p+1} \right\} \,,
\end{equation}

		\[
a_{k+1} = \frac{\lambda_{k+1}+\sqrt{\lambda_{k+1}^2+4\lambda_{k+1}A_k}}{2} 
\text{ , } 
A_{k+1} = A_k+a_{k+1}
\text{ , } 
\]
\[
\tilde{x}_k = \frac{A_k}{A_{k + 1}}y_k + \frac{a_{k+1}}{A_{k+1}} x_k. 
		\]
		\STATE $x_{k+1} := x_k-a_{k+1} \nabla f(y_{k+1}) - a_{k+1}\nabla g(y_{k+1})$.
		\ENDFOR
		\RETURN $y_{K}$ 
	\end{algorithmic}	
\end{algorithm}

Доказательство следующей теоремы см. в Приложении 1. Литературный обзор см. в препринте \cite{Kamzolov2020}.
\begin{theorem} \label{theoremCATD}
Пусть $y_k$ -- выход Алгоритма~\ref{alg:highorder} УМ($x_0$,$f$,$g$,$p$,$H$,$k$) после $k$ итераций при $p\geq 1$ и $H\ge (p+1)L_{p,f}$. Тогда
 \begin{equation} \label{speedCATD}
 F(y_k) - F(x_{\ast}) \leq \frac{c_p H R^{p+1}}{k^{\frac{3p +1}{2}}} \,,
 \end{equation}
где $c_p = 2^{p-1} (p+1)^{\frac{3p+1}{2}} / p!$,
$R=\|x_0 - x^{\ast}\|$. 

Более того, при $p \ge 2$ для достижения точности $\varepsilon$: $F(y_k) - F(x_{\ast}) \leq \varepsilon$ на каждой итерации УМ вспомогательную задачу \eqref{prox_step} придется перерешивать для подбора пары $(\lambda_{k+1},y_{k+1})$ не более чем $O\left(\ln\left(\varepsilon^{-1}\right)\right)$ раз.
\end{theorem}
Заметим, что приведенная выше теорема будет справедлива и при условии $H\ge 2L_{p,f}$ (независимо от $p \in \NN$). Это выводится из \eqref{eq_sumup}. Условие $H\ge (p+1)L_{p,f}$ было использовано, поскольку оно гарантирует выпуклость вспомогательной подзадачи \eqref{prox_step} \cite{NesterovImplementable}.  При этом условии и $g \equiv 0$ для $p = 1, 2, 3$ существуют эффективные способы решения вспомогательной задачи \eqref{prox_step} \cite{NesterovImplementable}. Для $p = 1$ существует явная формула для решения \eqref{prox_step}, для $p = 2, 3$ сложность \eqref{prox_step} такая же (с точностью до логарифмического по $\varepsilon$ множителя), как у итерации метода Ньютона \cite{NesterovImplementable}. 

Важно отметить, что вспомогательную задачу \eqref{prox_step} не обязательно решать точно: достаточно \cite{Kamzolov2020,Kamzolov2020Hyperfast} найти точку $\tilde{y}_{k+1}$, удовлетворяющую
\begin{equation}
\label{inexact1}
    \left\|\nabla \widetilde{\Omega}^k(\tilde{y}_{k+1}) \right\| \le \frac{1}{4p(p+1)}\|\nabla F(\tilde{y}_{k+1})\|.
\end{equation}
Такая модификация приведет лишь к появлению множителя $12/5$ в правой части \eqref{speedCATD}.

Будем говорить, что функция $F$ является $r$-равномерно выпуклой\\ (\dk{$p+1 \geq r \geq 2$}) с константой $\sigma_r > 0$, если
\begin{equation}\label{unif_conv}
    F(y)\geq F(x) + \la \nabla F(x), y-x \ra + \frac{\sigma_r}{r} \|y-x\|^r, \quad  x,y \in \R^d.
\end{equation}

В этом случае, используя \cite{Grapiglia2019}
\begin{equation}\label{inexact2}
  F(\tilde{y}_{k+1}) -   F(x_*) \le \frac{r-1}{r}\left(\frac{1}{\sigma_r}\right)^{\frac{1}{r-1}} \|\nabla F(\tilde{y}_{k+1})\|^{\frac{r}{r-1}},
\end{equation}
можно завязать критерий \eqref{inexact1} на желаемую точностью (по функции) решения исходной задачи $\varepsilon$ \cite{Kamzolov2020}:  $\left\|\nabla \widetilde{\Omega}^k(\tilde{y}_{k+1}) \right\|  = O\left(\left(\epsilon^{r-1}\sigma_r\right)^{\frac{1}{r}}\right)$. 

Более того, для $p = 1$ приведенные здесь выкладки можно уточнить, подчеркнув, тем самым, что сложность решения вспомогательной задачи может даже не зависеть от $\varepsilon$. Оказывается (см. \cite{Ivanova2019}), что условие
\begin{equation}
\label{inexact3}
\|\tilde{y}_{k+1} - y^*_{k+1}\| \le   \frac{H}{3H + 2L^g_1} \|\tilde{x}_{k} - y^*_{k+1}\|,
\end{equation}
где $y^*_{k+1}$ -- точное решение задачи \eqref{prox_step}, а $L^g_1$ -- константа Липшица градиента $\nabla g$, в теоретическом плане гарантирует то же, что и условие \eqref{inexact1} при $p=1$. А именно, Теорема~\ref{theoremCATD} останется верной с добавлением в правую часть \eqref{speedCATD} множителя $12/5$. 

Отметим, что оценка скорости сходимости
$\eqref{speedCATD}$ с точностью до числового множителя $c_p$ не может быть улучшена для класса выпуклых задач \eqref{eq1} с Липшицевой $p$-й производной и для широкого класса тензорных методов порядка $p$, описанном в \cite{NesterovImplementable}. 
При дополнительном предположении равномерной выпуклости $F$
оптимальный метод можно построить на базе УМ с помощью процедуры рестартов \cite{Kamzolov2020} -- см. Алгоритм~\ref{alg:restarts}.

\begin{algorithm} [h!]
\caption{Рестартованный УМ($x_0$,$f$,$g$,$p$,$r$,$\sigma_r$,$H$,$K$)}\label{alg:restarts}
	\begin{algorithmic}[1]
		\STATE \textbf{Input:} $r$-равномерно выпуклая функция $F = f + g : \R^d \rightarrow \R$ с константой $\sigma_r$ и УМ($x_0$,$f$,$g$,$p$,$H$,$K$).
		\STATE $z_0=x_0$.
		\FOR{$k = 0 $ \TO $K$}
		\STATE $R_k=R_0\cdot 2^{-k}$, 
		\begin{equation}
		\label{numberofrestarts}    
		N_k=\max \left\{ \left\lceil \left(\frac{r c_p H 2^r}{\sigma_r} R_k^{p+1-r}\right)^{\frac{2}{3p+1}}  \right\rceil, 1\right\}.
		\end{equation}
		\STATE  $z_{k+1} := y_{N_k}$, где $y_{N_k}$ -- выход УМ($z_k$,$f$,$g$,$p$,$H$,$N_k$).
		\ENDFOR
		\RETURN $z_{K}$ 
	\end{algorithmic}	
\end{algorithm}
\begin{theorem} \label{theoremRestartCATD}
Пусть $y_k$ -- выход Алгоритма~\ref{alg:restarts}  после $k$ итераций. Тогда если $H\ge (p+1)L_{p,f}$, $\sigma_r > 0$, то общее число вычислений \eqref{prox_step} для достижения $F(y_k) - F(x_*) \le \varepsilon$ будет:
\begin{equation}
    N = \tilde{O} \left(\left( \frac{H R^{p+1-r}}{\sigma_r} \right)^{\frac{2}{3p+1}}\right),
\end{equation}
где $\tilde{O}(\,)$ -- означает то же самое, что $O(\,)$ с точностью до множителя
$\ln\left({\varepsilon^{-1}}\right)$.
\end{theorem}
Все что было сказано после Теоремы~\ref{theoremCATD} можно отметить и в данном случае.

\section{Приложения}\label{section_3} 

\subsection{Ускоренные методы композитной оптимизации}\label{3.1} Если не думать о сложности решения подзадачи \eqref{prox_step}, например, считать, что $g$ какая-то простая функция и \eqref{prox_step} решается по явным формулам (как, например, для задачи LASSO), \ag{то} 
УМ описывает класс ускоренных методов (1, 2, 3, ... порядка) композитной оптимизации \cite{Gasnikov2018, NesterovLectures, Gasnikov2019}. При этом функция $g$ не обязана быть гладкой. В общем случае в строчке 5 Алгоритма~\ref{alg:highorder} под $\nabla g(y_{k+1})$ следует понимать такой субградиент функции $g$ в точке $y_{k+1}$, с которым субградиент правой части \eqref{prox_step} равен (близок) к нулю (немного переписав метод, от последнего ограничения можно отказаться). Отметим, что при $p = 1$ необходимость в поиске параметра $\lambda_{k+1}$ исчезает, что делает метод заметно проще.

\subsection{Ускоренные проксимальные методы. Каталист}\label{3.2} Если, считать $p = 1$, a $f\equiv 0$, $H > 0$ то получится ускоренный проксимальный метод. Отличительная особенность такого метода (см. также \cite{Ivanova2019}) от других известных ускоренных проксимальных методов заключается в том, что не требуется очень точно решать вспомогательную задачу. Критерий \eqref{inexact3} и сильная (2-равномерная) выпуклость вспомогательной подзадачи \eqref{prox_step} указывают на то, что сложность решения \eqref{inexact3} может не зависеть от желаемой точности решения исходной задачи $\varepsilon$. Таким образом, не теряется логарифмический множитель при использовании такой проксимальной оболочки для ускорения различных неускоренных процедур. Собственно, последнее направление получило название Каталист \cite{Catalyst}. До настоящего момента идея (Каталист) использования ускоренной проксимальной оболочки для `обертывания' неускоренных методов, решающих вспомогательную задачу \eqref{prox_step} на каждой итерации (при должном выборе параметра $H$), являлась наиболее общей идеей разработки ускоренных методов для разных задач. Мы получаем Каталист просто как частный случай УМ. Примемеры использования Каталист будут приведены в п.~\ref{3.4}.

\subsection{Разделение оракульных сложностей}
\label{3.3} Если считать, что для $g$ имеем $L_{p,g} < \infty$ (см. \eqref{def_lipshitz}) и  на вспомогательную задачу \eqref{prox_step} смотреть как на равномерно выпуклую достаточно гладкую задачу (с $f: = g$, $g(x) := 
\Omega_p\left(f,\tilde{x}_k;x\right) + 
\frac{(p+1)L_{p,f}}{\left( p+1\right)!}
\left\| x - \tilde{x}_k \right\|^{p+1}$), то для решения \eqref{prox_step}, в свою очередь, можно использовать Рестартованный УМ c $H \simeq (p+1)L_{p,g}$. В случае, когда $L_{p,f} \le L_{p,g}$ удается получить такие оценки сложности \cite{Kamzolov2020, Lan2019} (см. Теорему~\ref{theoremCATD}):
\begin{center}
$N_f = \tilde{O} \left(\left( \frac{L_{p,f} R^{p+1}}{\varepsilon} \right)^{\frac{2}{3p+1}}\right)$ -- число вызовов оракула для функции $f$,
\end{center}
\begin{center}
$N_g = \tilde{O} \left(\left( \frac{L_{p,g} R^{p+1}}{\varepsilon} \right)^{\frac{2}{3p+1}}\right)$ -- число вызовов оракула для функции $g$.
\end{center}
Вызов оракула подразумевает вычисление (старших) производных до порядка $p$ включительно. Таким образом, число вызовов оракула для каждой из функции $f$, $g$ является квазиоптимальным, т.е. оптимальным с точностью до логарифмического (от желаемой точности по функции) множителя. Аналогичные оценки можно получить и в $r$-равномерно ($r \ge 2$) выпуклом случае, см. п.~\ref{3.4}.

Заметим, что при $p = 1$ внутреннюю задачу \eqref{prox_step} не обязательно решать Рестартованным УМ. Можно использовать (ускоренные) покомпонентные и безградиентные методы, методы редукции дисперсии \cite{Gasnikov2018, Lan2019, Dvurechensky2017}. Причем ускорение можно получить из базовых неускоренных вариантов этих методов с помощью 
УМ (см. п.~\ref{3.2}). По сравнению с оболочкой, использованной в \cite{Ivanova2020}, УМ дает оценку сложности на логарифмический множитель лучше.
Это следует из теоретического анализа и было подтверждено в экспериментах \cite{code}. 

\subsection{Ускоренные методы для седловых задач} \label{3.4} Следуя, например \cite{Alkousa2019, Lin2020}, рассмотрим  выпукло-вогнутую седловую задачу
\begin{equation}\label{eq:3F}
\min _{x \in  \R^{d_x}} \{ F(x):=f(x)+\underbrace{\max _{y \in \R^{d_y}}\{G(x, y)-h(y)\}}_{g(x)=G(x, y^*(x))-h(y^*(x))} \;\; \} ,
\end{equation}
где $y^*(x) = \arg\max_{y \in \R^{d_y}}\{G(x, y)-h(y)\}$. Будем считать, что $\nabla f, \nabla G, \nabla h$ являются, соответственно, $L_f,L_G,L_h$-Липшицевыми. 
Также будем считать, что $f(x) + G(x,y)$ -- является $\mu_x$-сильно (2-равномерно) выпуклой по $x$, а $G(x, y)-h(y)$
-- $\mu_y$-сильно (2-равномерно) вогнутой по $y$. Тогда $F(x)$ будет $\mu_x$-сильно выпуклой, а $\nabla g$ будет $L_g = \left(L_{G} + 2L_G^2/\mu_y\right)$-Липшицевым \cite{Alkousa2019, Lin2020}. 

Если считать, что доступен $\nabla g$, то внешнюю задачу \eqref{eq:3F} можно решать ускоренным слайдингом (например, в варианте УМ с $p = 1$, см. п.~\ref{3.3})  за 
$\tilde{O}\left(\sqrt{L_f/\mu_x}\right)$ вычислений $\nabla f$ и $\tilde{O}\left(\sqrt{L_g/\mu_x}\right)$ вычислений $\nabla g$.  

Чтобы приближенно посчитать $\nabla g(x) = \nabla_x G(x,y^*(x))$ надо решить (с достаточной точностью) вспомогательную задачу в \eqref{eq:3F}, т.е. найти с нужной точностью $y^*(x)$. Это, в свою очередь, также можно сделать с помощью слайдинга (УМ с $p = 1$) за 
$\tilde{O}\left(\sqrt{L_h/\mu_y}\right)$ вычислений $\nabla h$ и $\tilde{O}\left(\sqrt{L_G/\mu_y}\right)$ вычислений $\nabla_y G$.

Резюмируя написанное, получаем, что исходную задачу \eqref{eq:3F} можно решить за $\tilde{O}\left(\sqrt{L_f/\mu_x}\right)$ вычислений $\nabla f$, $\tilde{O}\left(\sqrt{L_g/\mu_x}\right) \simeq \tilde{O}\left(\sqrt{L_G^2/(\mu_x\mu_y)}\right)$ вычислений $\nabla_x G$, $\tilde{O}\left(\sqrt{L_G^3/(\mu_x\mu_y^2)}\right)$ вычислений $\nabla_y G$, $\tilde{O}\left(\sqrt{L_h L_G^2/(\mu_x\mu_y^2)}\right)$ вычислений $\nabla h$. Поменяв порядок взятия $\min$ и $\max$ аналогичным образом можно прийти к оценкам $\tilde{O}\left(\sqrt{L_h/\mu_y}\right)$ вычислений $\nabla h$, $\tilde{O}\left(\sqrt{L_G^2/(\mu_x\mu_y)}\right)$ вычислений $\nabla_y G$, $\tilde{O}\left(\sqrt{L_G^3/(\mu_x^2\mu_y)}\right)$ вычислений $\nabla_x G$, $\tilde{O}\left(\sqrt{L_f L_G^2/(\mu_x^2\mu_y)}\right)$ вычислений $\nabla f$.

Оценки полученные на число вычислений $\nabla_x G$ и  $\nabla f$ в последнем случае не являются оптимальными \cite{Lin2020}. Чтобы улучшить данные оценки (сделать их оптимальными с точностью до логарифмических множителей \cite{Lin2020}) воспользуемся Каталистом, см. п.~\ref{3.2} (УМ, с $p = 1$, $H \gg \mu_x$, $f \equiv 0$, $g = F$, где $F$ определяется \eqref{eq:3F}). Если параметр метода $H$, то число итераций метода будет $\tilde{O}\left(\sqrt{H/\mu_x}\right)$, см. теорему~\ref{theoremRestartCATD}. На каждой итерации необходимо будет решать с должной точностью задачу вида \eqref{eq:3F}, в которой $L_f := L_f + H$, $\mu_x := \mu_x + H \simeq H$. Таким образом, для решения внутренней седловой задачи потребуется $\tilde{O}\left(\sqrt{L_h/\mu_y}\right)$ вычислений $\nabla h$, $\tilde{O}\left(\sqrt{L_G^2/(H\mu_y)}\right)$ вычислений $\nabla_y G$, $\tilde{O}\left(\sqrt{L_G^3/(H^2\mu_y)}\right)$ вычислений $\nabla_x G$, $\tilde{O}\left(\sqrt{(L_f + H) L_G^2/(H^2\mu_y)}\right)$ вычислений $\nabla f$. Считая для наглядности $L_f \ge L_G$ выберем $H = L_G$. Тогда итоговые оценки на число вычислений соответствующих градиентов будут такие: $\tilde{O}\left(\sqrt{L_h L_G/(\mu_x\mu_y)}\right)$ вычислений $\nabla h$, $\tilde{O}\left(\sqrt{L_G^2/(\mu_x\mu_y)}\right)$ вычислений $\nabla_y G$, $\tilde{O}\left(\sqrt{L_G^2/(\mu_x\mu_y)}\right)$ вычислений $\nabla_x G$, $\tilde{O}\left(\sqrt{L_f L_G/(\mu_x\mu_y)}\right)$ вычислений $\nabla f$.

За счет использования 
УМ приведенная выше схема улучшает похожую схему рассуждений из \cite{Lin2020} на логарифмический (по желаемой точности решения задачи) множитель, и обобщает ее на случай отличных от тождественного нуля функций $f$ и $h$.

Приведенная здесь схема рассуждений наглядно демонстрирует, как из одной универсальной схемы ускорения удается получить (`собрать как в конструкторе') 
оптимальный метод с точностью до логарифмического (по желаемой точности множителя (при $f \equiv 0$ и $h \equiv 0$ -- только при этих условиях известны нижние оценки \cite{Lin2020}).

\subsection{Сравнение с алгоритмом Монтейро--Свайтера}

Следуя \cite{Spokoiny2019}, рассмотрим следующую задачу оптимизации:
\begin{equation}
\min_{x \in \mathbb{R}^n} \{ F(x) := \underbrace{\log\left(\sum_{k=1}^p \exp \left(\langle A_k, x \rangle\right)\right)}_{= f(x)} + \underbrace{\frac{1}{2} \|G x\|_2^2}_{= g(x)}\;\; \},
\end{equation}
где $n=500$, $p=20000$, $A$ – разреженная $p \times n$ матрица с коэффициентом разреженности $0.001$ (под коэффициентом разреженности в данном случае понимается отношение числа ненулевых элементов матрицы к общему числу ее элементов), чьи ненулевые элементы есть независимые одинаково распределенные случайные величины из равномерного распределения $\mathcal{U}(-1, 1)$, а матрица $G^2$ получается из следующего выражения:
$$
G^2 = \sum_{i=1}^n \lambda_i \tilde{e}_i^\top \tilde{e}_i, 
$$
где $\sum_{i=1}^n \lambda_i = 1$ и $\left[\tilde{e}_i\right]_j \sim \mathcal{U}(1, 2)$ для каждой пары $i, j$.

Здесь $f$ имеет липшицев градиент с константой Липшица
$$
L_f = \max_{i = 1,...,n} \|A^{\langle k\rangle}\|_2^2,
$$
где за $A^{\langle k\rangle}$ обозначен $k$-й столбец матрицы $A$. 

На примере данной задачи сравним работу ускоренных методов, полученных с помощью алгоритма Монтейро--Свайтера \cite{Monteiro2013} ($L = 20 L_f$) и с помощью 
УМ ($H = L_f$) при использовании для решения вспомогательной задачи покомпонентного градиентного метода Нестерова (ПГМ) \cite{Nesterov2017} ($\beta=1/2$).

На рис. \ref{fig:ms-3d} покомпонентный метод, ускоренный с помощью алгоритма Монтейро--Свайтера, сравнивается с методом, ускоренным оболочкой УМ с различным числом итераций метода для решении вспомогательной задачи ($K_{\text{внутр}} = k$ соответствует $kn$ итерациям покомпонентного метода), и быстрым градиентным методом (БГМ). Представлен трехмерный график зависимости величины $(F(x_k) - F(x^*)) / (F(x_0) - F(x^*))$ (в log масштабе, $F(x^*)$ выбирается равным значению $F$ в точке, полученной после $25000$ итераций БГМ) от числа вычислений компонент градиентов $\nabla f_i$ и $\nabla g_i$, а также его двухмерные проекции. Так как некоторые методы требуют вычисления полного значения градиента ($\nabla f$ или $\nabla g$), обращения к оракулам, в таком случае, учитываются с весом $t_1 / t_2 \approx 2.5$, где $t_1$~--- среднее время вычисления полного градиента, $t_2$~--- среднее время вычисления одной компоненты. Как можно видеть из графиков, число обращений к оракулу $\nabla f_i$ ускоренного с помощью оболочки УМ метода меньше, чем у быстрого градиентного метода. Кроме того, оболочка УМ позволяет значительно сократить число обращений к оракулу $\nabla g_i$ по сравнению с алгоритмом Монтейро--Свайтера.

\begin{figure}[ht!]
    \centering
    \includegraphics[width=\textwidth]{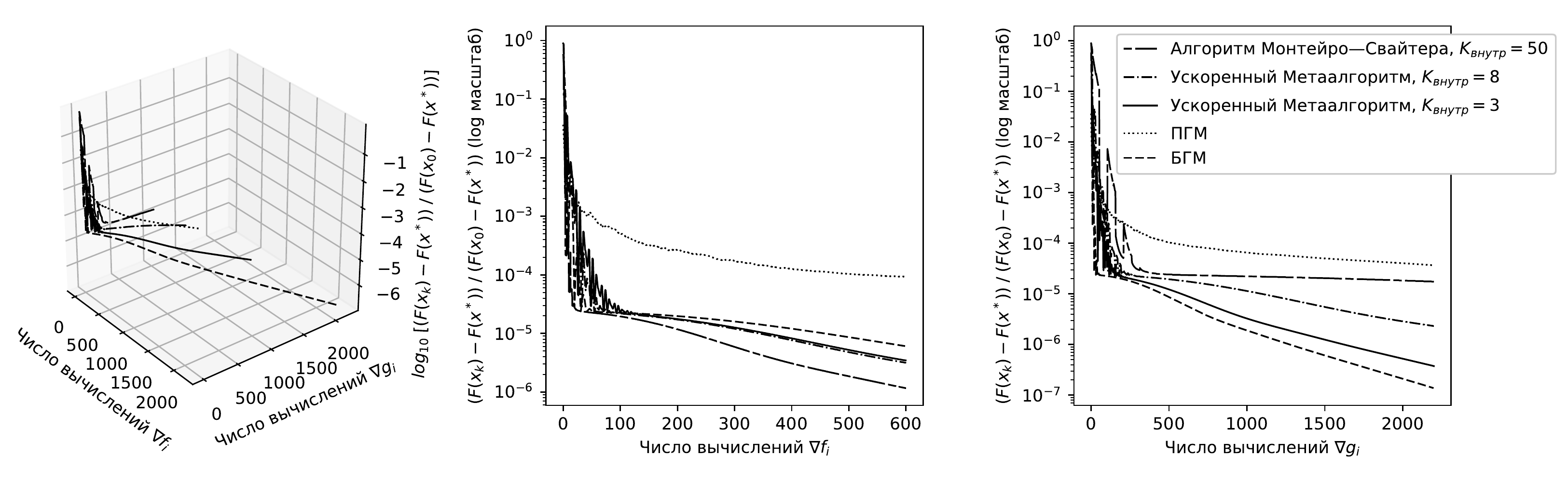}
    \vspace{-2em}
    \caption{График зависимости величины $(F(x_k) - F(x^*)) / (F(x_0) - F(x^*))$ (в log масштабе) от числа вычислений компонент градиентов $\nabla f_i$ и $\nabla g_i$. Двухмерные проекции}
    \label{fig:ms-3d}
\end{figure}

На рис. \ref{fig:ms} сравнивается работа методов в зависимости от времени работы и числа итераций внутреннего метода. Как можно видеть из графика \ref{fig:ms-time}, ускоренный с помощью оболочки УМ метод сходится по времени работы с большей скоростью, чем метод, ускоренный с помощью алгоритма Монтейро--Свайтера, а также с большей скоростью, чем быстрый градиентный метод.

\begin{figure}[ht]
  \vspace{-6em}   
  \subfloat{
	\begin{minipage}[c][1\width]{
	   0.5\textwidth}
	   \centering
	   \includegraphics[width=\textwidth]{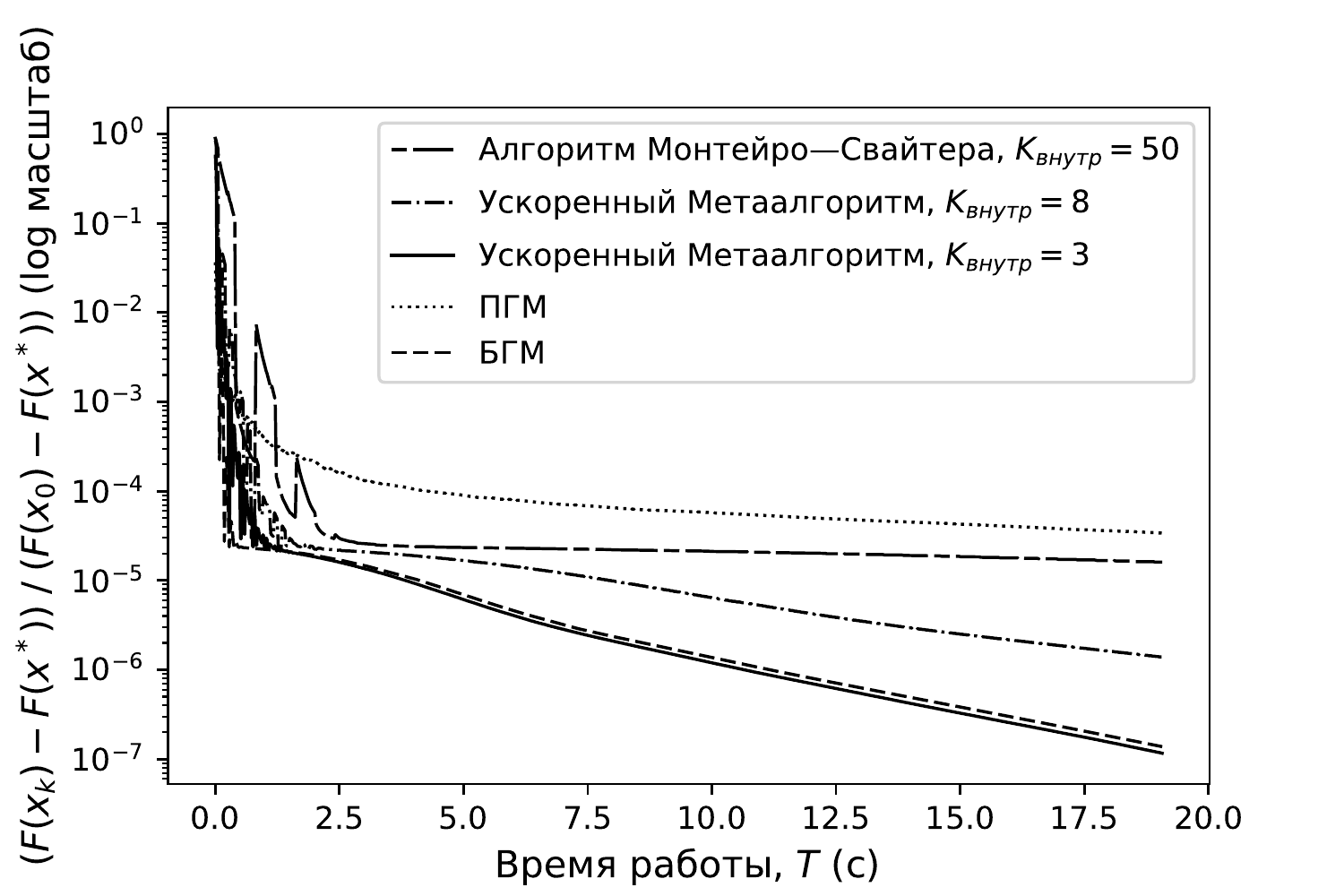}
        \vspace{-6em}
        \label{fig:ms-time}
	\end{minipage}}
 \hfill
  \subfloat{
	\begin{minipage}[c][1\width]{
	   0.5\textwidth}
	   \centering
	   \includegraphics[width=\textwidth]{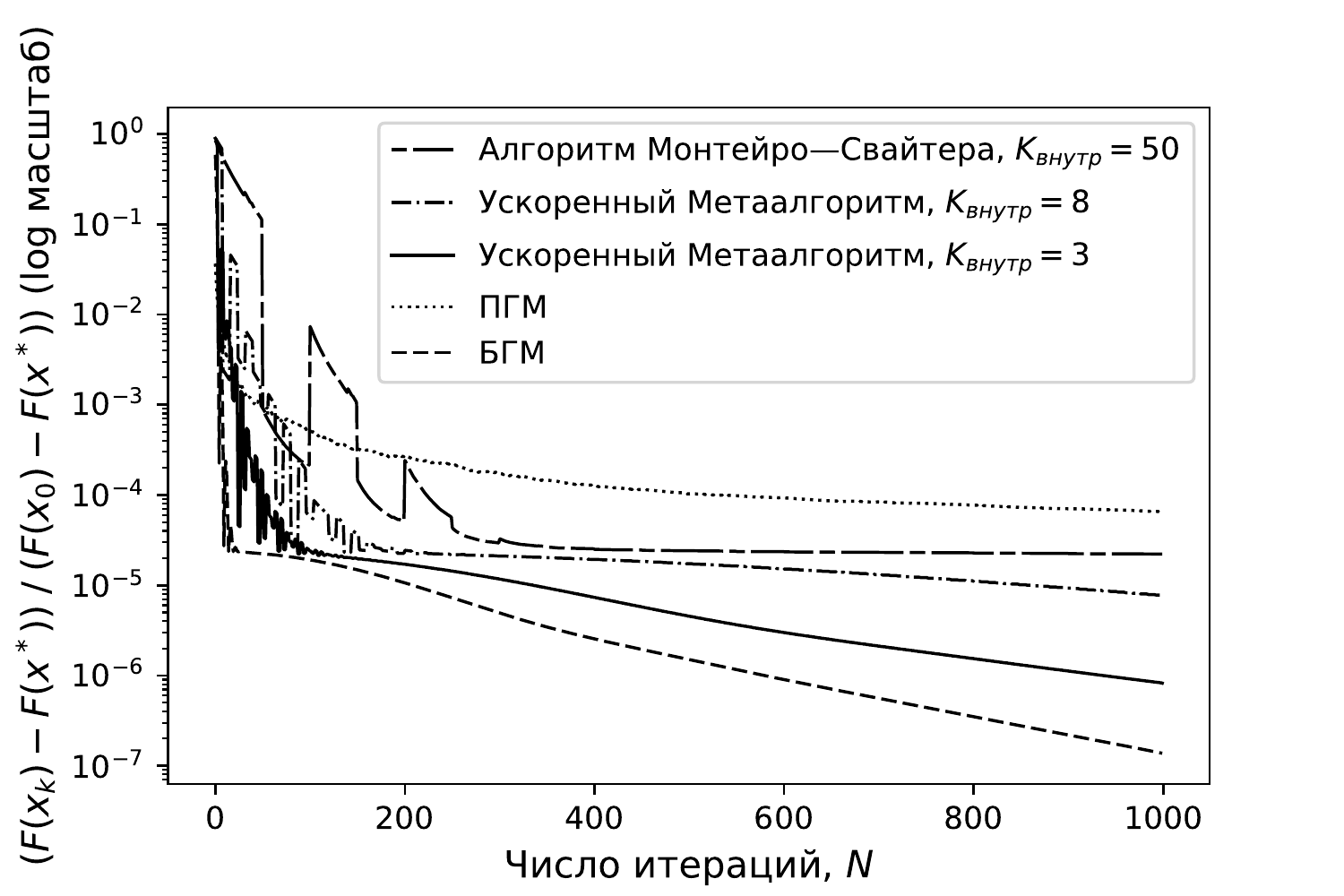}
        \vspace{-6em}
       \label{fig:ms-iters}
	\end{minipage}}
\caption{Графики зависимости величины $(F(x_k) - F(x^*)) / (F(x_0) - F(x^*))$ (в log масштабе) от времени работы и числа итераций внутреннего метода}
\label{fig:ms}
\end{figure}

\section{Возможные обобщения}\label{section_4}
Приводимые выше конструкции существенным образом базируются на том, что рассматриваются задачи безусловной оптимизации 
и используется евклидова норма. На данный момент открытым остается вопрос о перенесении  приведенных в статье результатов на задачи безусловной оптимизации с заменой евклиловой нормы на дивергенцию Брэгмана \cite{Gasnikov2018, Doikov2019}. Тем более открытым остается вопрос об использовании других (более общих) моделей в построении мажоратны целевой функции \eqref{eq_sumup} \cite{Gasnikov2018}. 

В работе \cite{Dvinskikh2020} было подмечено (в том числе и в модельной общности), что для получения оптимальных версий ускоренных алгоритмов для задач стохастической оптимизации нужно уметь оценивать как накапливается малый шум в градиенте в таких методах. Таким образом строятся ускоренные стохастические градиентные методы на базе ускоренных не стохастических (детерминированных) и конструкции, названной минибатчингом (замена градиента в детерминированном методе его оценкой, построеной на базе стохастических градиентов). Насколько нам известно, для тензорных методов вопрос построения ускоренных методов для задач стохастической оптимизации остается открытым. В частности, не известен ответ на такой вопрос: верно ли, что для задач сильно выпуклой стохастической оптимизации требования к точности аппроксимации старших производных с помощью минибатчинга снижаются по мере роста порядка производных, как это имеет место в невыпуклом случае \cite{Lucchi2019}? Для ответа на этот вопрос для тензорных методов также как и для градиентных $p = 1$ может пригодиться анализ чувствительности исследуемых методов к неточности в вычислении производных. Некоторый задел в этом направлении уже имеется \cite{Baes2009}. В частности, при $p = 1$ 
УМ 
демонстрирует стандартное для ускоренных методов накопление неточностей в градиенте \cite{Dvinskikh2020}.

Из статьи может показаться, что для безусловных достаточно гладких задач выпуклой оптимизации предлагаемый в статье подход дает возможность всегда строить <<оптимальные>> методы. На самом деле это не совсем так. Во-первых, построение оптимальных методов даже на базе 
\ag{одного только}
УМ может быть совсем не простой задачей, как показывает пример из п.~\ref{3.4}. 
Во-вторых, оговорка <<с точностью до логарифмических множителей>> весьма существенна. В частности, до сих пор остается открытым вопрос о том, устраним ли логарифмический мультипликативный зазор (по желаемой точности решения задачи по функции) между нижними оценками и тем, что дает УМ и другие ускоренные тензорные методы ($p \ge 2$), см. Теорему~\ref{theoremCATD}. В-третьих, упомянутые нижние оценки были получены для класса Крыловских методов (для тензорных методов чуть хитрее \cite{NesterovImplementable}), однако предлагаемая оболочка УМ в некоторых вариантах ее использования, в том числе в проксимальном варианте (Каталист), см. п.~\ref{3.2}, выводит из класса допустимых методов, для которого были полученые нижние оценки.

Работа А.В.~Гасникова поддержана грантом РФФИ 18-31-20005 мол\_а\_вед в п. 2, работа Д.И.~Камзолова поддержана грантом РФФИ 19-31-90170 Аспиранты в п. 3, работа П.Е.~Двуреченского поддержана грантом РФФИ 18-29-03071 мк в п. 3. Работа Д.М. Двинских и В.В. Матюхина была выполнена при поддержке Министерства науки и высшего образования Российской Федерации (госзадание) №075-00337-20-03, номер проекта 0714-2020-0005.

\section*{{\,\,\,\,\,\,\,\,\,\,\,\,\,\,\,\,\,\,\,\,\,\,\,\,\,\,\,\,\,\,\,\,\,\,\,\,\,\,\,\,\,\,\,\,\itПриложение 1} } \label{sec:apCATD}

В этом приложении представлено доказательство Теоремы \ref{theoremCATD}, основанное на доказательстве из статьи \cite{Bubeck2020}, с учетом добавления композитной функции. Следующая теорема базируется на Теореме 2.1 из \cite{Bubeck2020}
\begin{theorem} \label{thm:MS}
Пусть $(y_k)_{k \geq 1}$~--- это последовательность точек в $\R^d$, и  $(\lambda_k)_{k \geq 1}$~--- это последовательность в $\R_+$. Определим $(a_k)_{k \geq 1}$ такой, что $\lambda_k A_k = a_k^2$ и $A_k = \sum_{i=1}^k a_i$. Для любого $k\geq 0$ определим $x_k = x_0 - \sum_{i=1}^k a_i (\nabla f(y_i)+g'(y_i))$  и $\tilde{x}_k := \frac{a_{k+1}}{A_{k+1}} x_{k} + \frac{A_k}{A_{k+1}} y_k$. Также предположим, что если для некоторого $\sigma \in [0,1]$
\begin{equation} \label{eq:igdrefined}
\|y_{k+1} - (\tilde{x}_k - \lambda_{k+1} \nabla f(y_{k+1}))\| \leq \sigma \cdot \|y_{k+1} - \tilde{x}_k\| \,,
\end{equation}
тогда для любого $x \in \R^d$ верны неравенства:
\begin{equation} \label{eq:rate}
F(y_k) - F(x) \leq \frac{2 \|x\|^2}{\left(\sum_{i=1}^k \sqrt{\lambda_i} \right)^2}  \,,
\end{equation}
и
\begin{equation} \label{eq:Alambdatradeoff}
\sum_{i=1}^k \frac{A_i}{\lambda_i} \|y_i - \tilde{x}_{i-1}\|^2 \leq \frac{\|x^*\|^2}{1-\sigma^2} \,.
\end{equation}
\end{theorem}
Для доказательства этой теоремы мы введем дополнительные леммы, основанные на леммах 2.2-2.5 и 3.1 из \cite{Bubeck2020}, леммы 2.6 и 3.3 могут использоваться без изменений.

\begin{lemma} \label{lem:basic1}
Пусть $\psi_0(x) = \frac{1}{2} \|x-x_0\|^2$, и по индукции определим $\psi_{k}(x) = \psi_{k-1}(x) + a_{k} \Omega_1(F, y_{k}, x)$, тогда $x_k =x_0 - \sum_{i=1}^k a_i (\nabla f(y_i) + g'(y_i))$~--- это минимизатор функции $\psi_k$, и верно
$\psi_k(x) \leq A_k F(x) + \frac{1}{2} \|x-x_0\|^2$, где $A_k = \sum_{i=1}^k a_i$. 
\end{lemma}

\begin{lemma} \label{lem:basic2}
Пусть $z_k$ такая, что 
\begin{equation} \label{eq:tosatisfy}
\psi_k(x_k) - A_k F(z_k) \geq 0 \,.
\end{equation}
Тогда для любого $x$,
\begin{equation} \label{eq:tosatisfy2}
F(z_k) \leq F(x) + \frac{\|x-x_0\|^2}{2 A_k} \,.
\end{equation}
\end{lemma}

\begin{proof}
Из Леммы \ref{lem:basic1}) можно получить, что
\[
A_k F(z_k) \leq \psi_k(x_k) \leq \psi_k(x) \leq A_k F(x) + \frac{1}{2}\|x-x_0\|^2 \,.
\]
\end{proof}

\begin{lemma} \label{lem:basic3}
Для любого $x$ верно следующее неравенство
\begin{align*}
& \psi_{k+1}(x) - A_{k+1} F(y_{k+1}) - (\psi_k(x_k) - A_k F(z_k)) \\
& \geq A_{k+1} (\nabla f(y_{k+1}) +g'(y_{k+1}))\cdot \left(\frac{a_{k+1}}{A_{k+1}} x + \frac{A_k}{A_{k+1}} z_k - y_{k+1} \right ) + \frac{1}{2} \|x -x_k\|^2 \,.
\end{align*}
\end{lemma}

\begin{proof}
Во-первых, простыми вычислениями получим
\[
\psi_k(x) = \psi_k(x_k) + \frac{1}{2} \|x- x_k\|^2,\] и \[ \psi_{k+1}(x) = \psi_k(x_k) + \frac{1}{2} \|x-x_k\|^2 + a_{k+1} \Omega_1(f, y_{k+1}, x) \,,
\]
таким образом
\begin{equation} \label{eq:ind1}
\psi_{k+1}(x) - \psi_k(x_k) = a_{k+1} \Omega_1(F, y_{k+1}, x) + \frac{1}{2} \|x-x_k\|^2 \,.
\end{equation}
Теперь мы хотим, чтобы $A_{k+1} F(z_{k+1}) - A_k F(z_k)$ было нижней оценкой неравенства \eqref{eq:ind1}, когда вычисляем $x=x_{k+1}$. 
Используя
 $\Omega_1(F, y_{k+1}, z_k) \leq f(z_k),$
 мы получаем:
\begin{eqnarray*}
& &a_{k+1}\Omega_1(F, y_{k+1}, x)  =  A_{k+1} \Omega_1(F, y_{k+1}, x) - A_k \Omega_1(F, y_{k+1}, x) \\
& = & A_{k+1} \Omega_1(F, y_{k+1}, x) - A_k \nabla F(y_{k+1}) \cdot (x - z_k) - A_k \Omega_1(F, y_{k+1}, z_k) \\
 & = & A_{k+1} \Omega_1\left(F, y_{k+1}, x - \frac{A_k}{A_{k+1}} (x - z_k) \right ) - A_k \Omega_1(F, y_{k+1}, z_k) \\
 & \geq & A_{k+1} F(y_{k+1}) - A_k F(z_k)\\
 &+ &A_{k+1} (\nabla f(y_{k+1})+g'(y_{k+1})) \cdot \left(\frac{a_{k+1}}{A_{k+1}} x + \frac{A_k}{A_{k+1}} z_k - y_{k+1} \right ) \,,
\end{eqnarray*}
что завершает доказательство.
\end{proof}

\begin{lemma} \label{lem:basic4}
Обозначим $\lambda_{k+1} := \frac{a_{k+1}^2}{A_{k+1}}$ и $\tilde{x}_k := \frac{a_{k+1}}{A_{k+1}} x_{k} + \frac{A_k}{A_{k+1}} y_k$, и получим:
\begin{align*}
& \psi_{k+1}(x_{k+1}) - A_{k+1} F(y_{k+1}) - (\psi_k(x_k) - A_k F(y_k)) \\
& \geq \frac{A_{k+1}}{2 \lambda_{k+1}} \bigg( \|y_{k+1} - \tilde{x}_k\|^2 - \|y_{k+1} - (\tilde{x}_k - \lambda_{k+1} (\nabla f(y_{k+1}))+g'(y_{k+1})) \|^2 \bigg) \,.
\end{align*}
А применив дополнительно неравенство \eqref{eq:igdrefined}, получаем
$$\psi_{k}(x_{k})-A_{k}F(y_{k})\geq\frac{1-\sigma^{2}}{2}\sum_{i=1}^{k}\frac{A_{i}}{\lambda_{i}}\|y_{i}-\tilde{x}_{i-1}\|^{2}.$$
\end{lemma}

\begin{proof}
Используем Лемму \ref{lem:basic3} при $z_k = y_k$ и $x=x_{k+1}$ и получим, что
(при $\tilde{x} := \frac{a_{k+1}}{A_{k+1}} x + \frac{A_k}{A_{k+1}} y_k$): 
\begin{align*}
& (\nabla f(y_{k+1})+g'(y_{k+1})) \cdot \left(\frac{a_{k+1}}{A_{k+1}} x + \frac{A_k}{A_{k+1}} y_k - y_{k+1} \right )  + \frac{1}{2 A_{k+1}} \|x - x_k\|^2 \\
& = (\nabla f(y_{k+1})+g'(y_{k+1})) \cdot (\tilde{x} - y_{k+1}) + \frac{1}{2 A_{k+1}} \left\|\frac{A_{k+1}}{a_{k+1}} \left(\tilde{x} - \frac{A_k}{A_{k+1}} y_k \right) - x_k \right\|^2 \\
& = (\nabla f(y_{k+1})+g'(y_{k+1})) \cdot (\tilde{x} - y_{k+1}) + \frac{A_{k+1}}{2 a_{k+1}^2} \left\|\tilde{x} - \left(\frac{a_{k+1}}{A_k} x_k + \frac{A_k}{A_{k+1}} y_k \right) \right\|^2 \,.
\end{align*}
Откуда следует следущее неравенство:
\begin{align*}
& \psi_{k+1}(x_{k+1}) - A_{k+1} F(y_{k+1}) - (\psi_k(x_k) - A_k F(y_k)) \\
& \geq A_{k+1} \cdot \min_{x \in \R^d} \left\{ (\nabla f(y_{k+1})+g'(y_{k+1})) \cdot (x - y_{k+1}) + \frac{1}{2 \lambda_{k+1}} \|x - \tilde{x}_k\|^2 \right\} \,.
\end{align*}
Значиение минимума можно лего посчитать.
\end{proof}

Для первого выражения Теоремы \ref{thm:MS} достаточно объединить Лемму \ref{lem:basic4} с Леммой \ref{lem:basic2} и Леммой 2.5 из \cite{Bubeck2020}.
Второе выражение в Теореме \ref{thm:MS} следует из Леммы \ref{lem:basic4} и Леммы \ref{lem:basic1}.

Следующая лемма доказывает, что минимизация ряда Тэйлора порядка $p$ для \eqref{prox_step} может быть представлено, как неявный градиентный шаг для некоторого большого размера шага.
\begin{lemma} \label{lem:controlstepsize}
Неравенство \eqref{eq:igdrefined} верно при $\sigma = 1/2$ для \eqref{prox_step}, из этого следует, что: 
\begin{equation} \label{eq:key4}
\frac{1}{2} \leq \lambda_{k+1} \frac{L_p \cdot \|y_{k+1} - \tilde{x}_k\|^{p-1}}{(p-1)!}  \leq \frac{p}{p+1} \,.
\end{equation}
\end{lemma}

\begin{proof}
Из условия оптимальности следует, что
\begin{equation} \label{eq:KKT_TD}
\nabla_y f_p(y_{k+1}, \tilde{x}_k) + \frac{L_p \cdot (p+1)}{p!} (y_{k+1} - \tilde{x}_k) \|y_{k+1} - \tilde{x}_k\|^{p-1} + g'(y_{k+1})= 0 \,. 
\end{equation}
Откуда следует
\begin{align*}
&y_{k+1} - (\tilde{x}_k - \lambda_{k+1} (\nabla f(y_{k+1})+ g'(y_{k+1})))  = \lambda_{k+1} (\nabla f(y_{k+1})+ g'(y_{k+1}))\\
&- \frac{p!}{L_p \cdot (p+1) \cdot \|y_{k+1} - \tilde{x}_k\|^{p-1}} (\nabla_y f_p(y_{k+1}, \tilde{x}_k)+g'(y_{k+1})) \,.
\end{align*}
Используя ряд Тэйлора для градиента функции получаем:
\[
\|\nabla f(y) - \nabla_y f_p(y, x)\| \leq \frac{L_p}{p!} \|y - x\|^p \,,
\]
таким образом
\begin{align*}
& \|y_{k+1} - (\tilde{x}_k - \lambda_{k+1} (\nabla f(y_{k+1})+ g'(y_{k+1}))) \| \\
& \leq \lambda_{k+1} \frac{L_p}{p!} \|y_{k+1} - \tilde{x}_k\|^p  \\
& + \left|\lambda_{k+1} - \frac{p!}{L_p \cdot (p+1) \cdot \|y_{k+1} - \tilde{x}_k\|^{p-1}} \right| \cdot \|\nabla_y f_p(y_{k+1}, \tilde{x}_k)+ g'(y_{k+1})\| \\
& \leq \|y_{k+1} - \tilde{x}_k\|  \\
&\cdot\left(\lambda_{k+1} \frac{L_p}{p!} \|y_{k+1} - \tilde{x}_k\|^{p-1} + \left|\lambda_{k+1}\frac{L_p \cdot (p+1) \cdot \|y_{k+1} - \tilde{x}_k\|^{p-1}}{p!} - 1\right|  \right) \\
&=\|y_{k+1}-\tilde{x}_{k}\|\left(\frac{\eta}{p}+\left|\eta\cdot\frac{p+1}{p}-1\right|\right),
\end{align*}
где мы используем \eqref{eq:KKT_TD} во втором неравенстве, и предполагаем \\$\eta := \lambda_{k+1} \frac{L_p \cdot \|y_{k+1} - \tilde{x}_k\|^{p-1}}{(p-1)!}$ в последнем равенстве.  Итоговый результат получаем из предположения, что $1/2 \leq \eta \leq p/(p+1)$ в \eqref{eq:key4}. 
\end{proof}

В заключении, если мы заменим $\|x^{\ast}\|$ на $\|x_0-x^{\ast}\|$ в Лемме 3.3 и используем Лемму 3.4 из \cite{Bubeck2020}, то мы получаем доказательство Теоремы \ref{theoremCATD}. 

\section*{{\,\,\,\,\,\,\,\,\,\,\,\,\,\,\,\,\,\,\,\,\,\,\,\,\,\,\,\,\,\,\,\,\,\,\,\,\,\,\,\,\,\,\,\,\itПриложение 2} } \label{sec:CATDproff}
В данной секции мы докажем Теорему \ref{theoremRestartCATD}.
\begin{proof}
Так как функция $F$ является $r$-равномерно выпуклой, то мы получаем
\begin{align*}
    R_{k+1}&=\|z_{k+1}-x_{\ast}\| \leq \left( \frac{r \left( F(z_{k+1})-F(x_{\ast}) \right)}{\sigma_r} \right)^{\frac{1}{r}}
    \overset{\eqref{speedCATD}}{\leq} \left( \frac{r \left( \frac{c_p L_p R_{k}^{p+1}}{N_k^{\frac{3p+1}{2}}} \right)}{\sigma_r} \right)^{\frac{1}{r}}\\ &=\left( \frac{r c_p L_p R_{k}^{p+1}}{\sigma_r N_k^{\frac{3p+1}{2}}} \right)^{\frac{1}{r}}
    \overset{\eqref{numberofrestarts}}{\leq} \left( \frac{ R_{k}^{p+1}}{2^r R_k^{p+1-r} }\right)^{\frac{1}{r}} = \frac{ R_{k}}{2}.
\end{align*}
Теперь вычислим общее чилсло шагов метода \ref{alg:highorder}.
\begin{align*}
    \sum\limits_{k=0}^K N_k &\leq \sum\limits_{k=0}^K \left( \frac{r c_p L_p 2^r}{\sigma_r} R_k^{p+1-r} \right)^{\frac{2}{3p+1}}+K\\
    & = \sum\limits_{k=0}^K \left( \frac{r c_p L_p 2^r}{\sigma_r} (R_0 2^{-k})^{p+1-r} \right)^{\frac{2}{3p+1}}+K\\
    &=\left(\frac{r c_p L_p 2^r R_0^{p+1-r}}{\sigma_r}\right)^{\frac{2}{3p+1}} \sum\limits_{k=0}^K  2^{\frac{-2(p+1-r)k}{3p+1}}+K.
\end{align*}
\end{proof}

\end{document}